\documentclass[11pt]{article}%
\usepackage{amsmath}
\usepackage{amsfonts}
\usepackage{amssymb}
\usepackage{graphicx}%
\providecommand{\U}[1]{\protect\rule{.1in}{.1in}}
\newtheorem{theorem}{Theorem}

\newtheorem{corollary}[theorem]{Corollary}

\newtheorem{lemma}[theorem]{Lemma}

\newtheorem{proposition}[theorem]{Proposition}

\newenvironment{proof}[1][Proof]{\noindent\textbf{#1.} }{\ \rule{0.5em}{0.5em}}
\begin{document}

\title{ Gauss map and the topology of constant mean curvature hypersurfaces of
$\mathbb{S}^{7}$ and $\mathbb{CP}^{3}$}
\author{Fidelis Bittencourt
\and Pedro Fusieger
\and Eduardo R. Longa
\and Jaime Ripoll}
\date{}
\maketitle

\begin{abstract}
We define a Gauss map $\gamma:M\rightarrow\mathbb{S}^{6}$ of an oriented
hypersurface $M$ of the unit sphere $\mathbb{S}^{7}$ and prove that $\gamma$
is harmonic if and only if $M$ has CMC. Results on the geometry and topology
of CMC hypersurfaces of $\mathbb{S}^{7}$, under hypothesis on the image of
$\gamma$, are then obtained. By a Hopf symmetrization process we define a
Gauss map for hypersurfaces of $\mathbb{CP}^{3}$ and obtain similar results
for CMC hypersurfaces of this space.

\end{abstract}

\section{Introduction}

\qquad The image of the Gauss map of a minimal surface in the Euclidean space
is a classical topic of study on Differential Geometry. A well known result
says that if the Gauss image of a complete minimal surface of $\mathbb{R}^{3}$
lies in a hemisphere of the sphere then the surface is a plane (see
\cite{BdoC}). This result was extended to constant mean curvature
(CMC)\ surfaces by D. Hoffman, R. Osserman and R. Schoen, who proved that if
the Gauss image of a complete CMC surface of $\mathbb{R}^{3}$ is contained in
a closed hemisphere of the sphere then the surface is a plane or a right
circular cylinder (\cite{HOS}). An important ingredient used in \cite{HOS},
due to Ruh-Vilms, asserts that a surface has CMC if and only if its Gauss map
is harmonic (the statement of Ruh-Vilms' theorem is more general, see
\cite{RV}). Several works extending and generalizing the Euclidean Gauss map
to higher dimensions and to different ambient spaces have appeared during the
last decades. However, to the best of authors' knowledge, extensions or
generalizations of Ruh-Vilms' theorem were only investigated on the papers
\cite{BR}, \cite{EFR}, \cite{FM}, \cite{JR}, \cite{RR}.

In this paper we use the multiplication $\cdot$ of the unit sphere
$\mathbb{S}^{7}$ centered at the origin of $\mathbb{R}^{8}$ induced by a
normed division algebra structure of $\mathbb{R}^{8}$ $($the so called
octonionic multiplication, see $\cite{Ba}),$ to define a translation map
$\Gamma:T\mathbb{S}^{7}\rightarrow T_{1}\mathbb{S}^{7}$ by $\Gamma
_{x}(v):=x^{-1}\cdot v$, $(x,v)\in T\mathbb{S}^{7}$, where $T\mathbb{S}^{7}$
is the tangent bundle of $\mathbb{S}^{7}$ and $1$ is the unit of $\cdot$. The
map $\Gamma$ provides a simple and natural definition of a Gauss map
$\gamma:M\rightarrow\mathbb{S}^{6}\subset T_{1}\mathbb{S}^{6}$ of an
orientable hypersurface $M$ of $\mathbb{S}^{7}$. Indeed, if $\eta$ is a unit
normal vector for $M$, then we set $\gamma(x):=\Gamma_{x}(\eta(x))$.

We obtain an extension to $\gamma$ of a well known formula for the Laplacian
of the Euclidean Gauss map, and prove an extension of Ruh-Vilms' theorem to a
hypersurface $M$ of $\mathbb{S}^{7}$: $\gamma:M\rightarrow\mathbb{S}^{6}$ is
harmonic if and only if $M$ is a CMC hypersurface (Theorem \ref{for}). We then
use formula (\ref{f}) for the Laplacian of $\gamma$ to obtain results on the
topology of a CMC hypersurface of $\mathbb{S}^{7}$ under certain conditions on
the image of $\gamma$ (Theorem \ref{mao}). These constructions are only
possible in the $7$ dimensional sphere (besides $\mathbb{S}^{3}$, studied in
\cite{EFR}) since, by a classical result of Hurwitz, the only normed division
algebra structures of $\mathbb{R}^{n}$ are $\mathbb{R}$, $\mathbb{C}$,
$\mathbb{R}^{4}$ $($the quaternions) and $\mathbb{R}^{8}$ $\mathbb{(}$the
octonions). By a process of symmetrization we introduce, from the translation
of $\mathbb{S}^{7}$, a translation on $\mathbb{CP}^{3}$ minus a totally
geodesic $\mathbb{CP}^{2},$ obtaining similar results on $\mathbb{CP}%
^{3}\backslash\mathbb{CP}^{2}.$

\section{The octonionic Gauss map and constant mean curvature hypersurfaces of
$\mathbb{S}^{7}$}

\qquad As mentioned in the introduction, we use here the multplication $\cdot$
on $\mathbb{S}^{7},$ called octonionic multiplication, induced by the unique
normed division algebra structure of $\mathbb{R}^{8}$ (we note that with this
operation $\mathbb{S}^{7}$ is not a Lie group since $\cdot$ is not
associative. See \cite{Ba}). We define a translation $\Gamma:T\mathbb{S}%
^{7}\rightarrow T_{1}\mathbb{S}^{7}$ by $\Gamma(x,v)=x^{-1}\cdot v,$ $(x,v)\in
T\mathbb{S}^{7},$ where $1\in\mathbb{S}^{7}$ is the neutral element of
$\cdot.$

Any vector $v\in T_{x_{0}}\mathbb{S}^{7}$ at any given point $x_{0}%
\in\mathbb{S}^{7m}$ determines a vector field $V_{v}$ on $\mathbb{S}^{7}$,
which we call here a \emph{translational vector field, }given by
\[
V_{v}(x)=\Gamma_{x^{-1}}(\Gamma_{x_{0}}(v))
\]
or
\[
V_{v}(x)=x\cdot(x_{0}^{-1}\cdot v)=:R_{\left(  x_{0}^{-1}\cdot v\right)
}(x),\quad x\in\mathbb{S}^{7}.
\]

A \emph{Hopf} vector field in $\mathbb{S}^{7}$ is any vector field which is
$\operatorname*{Ad}$-conjugated to the vector field of the form%

\[
X_{0}=\left(
\begin{array}
[c]{cccc}%
A & 0 & 0 & 0\\
0 & A & 0 & 0\\
0 & 0 & A & 0\\
0 & 0 & 0 & A
\end{array}
\right)  ,\text{ }A=\left(
\begin{array}
[c]{cc}%
0 & 1\\
-1 & 0
\end{array}
\right)
\]
that is, a vector field of the form $X=TX_{0}T^{-1}$ with $T\in O(8)$.

\begin{lemma}
\label{kil} A translational vector field of $\mathbb{S}^{7}$ is a multiple of
a Hopf vector field of $\mathbb{S}^{7}$. In particular, translational vector
fields are Killing fields without singularities.
\end{lemma}

\begin{proof}
It is easy to see that a translational vector field $X$ is of the form
$X(x)=x\cdot v$, for some $v\in T_{1}\mathbb{S}^{m}$. Assuming that $X$ is
nonzero, then $v\neq0$ and we have $X=\left\Vert v\right\Vert R_{v/\left\Vert
v\right\Vert }$. Since $\left\Vert v/\left\Vert v\right\Vert \right\Vert =1$
and $v/\left\Vert v\right\Vert \in T_{1}\mathbb{S}^{m}$ it follows that
$R_{v/\left\Vert v\right\Vert }$ is skew-symmetric and orthogonal (see Section
2.3 of \cite{Ba}). The eigenvalues of $X$ are therefore pure imaginary and
must be all equal to $i$. The lemma then follows from elementary Linear Algebra.
\end{proof}

Let $M$ be an oriented immersed hypersurface of $\mathbb{S}^{7}$ and let
$\eta:M\rightarrow T\mathbb{S}^{7}$ be a unit normal vector field along $M$.
We define the Gauss map (octonionic Gauss map)
\[
\gamma:M\rightarrow\mathbb{S}^{6}\subset T_{1}\mathbb{S}^{7}%
\]
of $M$, where $\mathbb{S}^{6}$ is the unit sphere centered at the origin of
$T_{1}\mathbb{S}^{7}$, by
\[
\gamma(x):=\Gamma_{x}(\eta(x))=x^{-1}\cdot\eta(x),\quad x\in M.
\]

The Laplacian $\Delta\gamma$ of $\gamma$ is defined by setting
\[
\Delta\gamma:=\sum_{j=1}^{6}\Delta_{M}\left(  \left\langle \gamma
,v_{j}\right\rangle \right)  v_{j}%
\]
where $\Delta_{M}$ is the Laplacian in $M$ and $\{v_{1},\dots,v_{6}\}$ is an
orthonormal basis of $T_{1}\mathbb{S}^{6}$. It is easy to see that
$\Delta\gamma$ does not depend on the choice of basis. One may prove that
$\gamma$ is harmonic if and only if $\left(  \Delta\gamma\right)  ^{\top}=0$,
where
\[
\left(  \Delta\gamma\right)  ^{\top}=\Delta\gamma-\left\langle \Delta
\gamma,\gamma\right\rangle \gamma
\]
is the orthogonal projection of $\Delta\gamma$ onto the tangent spaces of
$\mathbb{S}^{6}$ (see \cite{EL}).

\begin{theorem}
\label{for} Let $M$ be an oriented immersed hypersurface of $\mathbb{S}^{7}$,
and let $\gamma:M\rightarrow\mathbb{S}^{6}\subset T_{1}\mathbb{S}^{7}$ be the
Gauss map of $M$ determined by a unit normal vector field $\eta$ along $M$.
Then
\begin{equation}
\Delta\gamma(x)=-6\Gamma_{x}\left(  \operatorname{grad}H(x)\right)  -\left(
6+\left\Vert A(x)\right\Vert ^{2}\right)  \gamma(x),\quad x\in M, \label{f}%
\end{equation}
where $H:M\rightarrow\mathbb{R}$ is the mean curvature function of $M$ with
respect to $\eta$, $\operatorname{grad}H$ is its gradient on $M$, $A$ is the
second fundamental form of $M$ and $\left\Vert A\right\Vert $ is its norm.
\end{theorem}

\begin{proof}
Let $\{v_{1},\dots,v_{m}\}$ be an orthonormal basis of $T_{1}\mathbb{S}^{m}$.
Since, by Lemma \ref{kil}, the vector fields $V_{v_{j}}$ are Killing fields,
we may use Proposition 1 of \cite{FR} to obtain, at any $x\in M$,
\begin{align*}
\Delta\gamma(x) &  =\sum_{j=1}^{m}\Delta\left(  \left\langle \gamma
,v_{j}\right\rangle \right)  (x)v_{j}=\sum_{j=1}^{m}\Delta\left(  \left\langle
\eta,V_{v_{j}}\right\rangle \right)  (x)v_{j}\\
&  =-6\sum_{j=1}^{m}\left\langle V_{v_{j}},\operatorname{grad}H\right\rangle
(x)v_{j}-\left(  6+\left\Vert A\right\Vert ^{2}\right)  \sum_{j=1}%
^{m}\left\langle V_{v_{j}},\eta\right\rangle (x)v_{j}\\
&  =-6\sum_{j=1}^{m}\left\langle v_{j},\Gamma_{x}\left(  \operatorname{grad}%
H(x)\right)  \right\rangle v_{j}-\left(  6+\left\Vert A\right\Vert
^{2}\right)  \sum_{j=1}^{m}\left\langle v_{j},\gamma(x)\right\rangle v_{j}\\
&  =-6\Gamma_{x}\left(  \operatorname{grad}H(x)\right)  -\left(  6+\left\Vert
A\right\Vert ^{2}\right)  \gamma(x),
\end{align*}
proving the theorem.
\end{proof}

\begin{corollary}
\label{harm} Let $M$ be an oriented immersed hypersurface of $\mathbb{S}^{7}$
and let $\gamma:M\rightarrow\mathbb{S}^{6}$ be the Gauss map of $M$. Then $M$
has constant mean curvature if and only if $\gamma$ is harmonic. Equivalently,
$M$ has constant mean curvature if and only, for any $v\in T_{1}\mathbb{S}%
^{7},$ the function $x\mapsto\left\langle \gamma(x),v\right\rangle $ is
harmonic in $M.$
\end{corollary}

A Gauss map $\gamma$ of an orientable hypersurface $M$ of $\mathbb{S}^{n},$
$n\geq3,$ that have been often studied in the literature associates, to each
$x\in M,$ the vector $\eta(x)\in\mathbb{S}^{n},$ where $\eta$ is a unit normal
vector field along $M.$ E. De Giorgi (\cite{DG}) and J. Simons (\cite{JS})
proved that if $M$ is compact, has constant mean curvature and $\gamma(M)\ $is
contained in an \emph{open} hemisphere of $\mathbb{S}^{n}$ then $M$ must be a
totally geodesic hypersphere of $\mathbb{S}^{n}.$ Using the octonionic Gauss
map we obtain:

\begin{theorem}
Let $M$ be an orientable compact CMC hypersurface of $\mathbb{S}^{7}$ and let
$\gamma:M\rightarrow\mathbb{S}^{6}$ be its (octonionic) Gauss map. Then
$\gamma(M)$ is not contained in any open hemisphere of $\mathbb{S}^{6}.$
\end{theorem}

\begin{proof}
Assume that $\gamma(M)$ is contained in a hemisphere of $\mathbb{S}^{6}$
having $v\in\mathbb{S}^{6}$ as pole. This means that $\left\langle
\gamma(x),v\right\rangle >0$ for all $x\in M.$ Since $M$ is compact, we may
find $7$ linearly independent vector $v_{1},...,v_{7}\in T_{1}\mathbb{S}^{7}$
such that $\left\langle \gamma(x),v_{i}\right\rangle >0,$ for all $x\in M$ and
$1\leq i\leq7.$ The functions $\left\langle \gamma(x),v_{i}\right\rangle $
being harmonic and not changing sign must be constant. Hence $\gamma$ must be
a constant map and then $\Delta\gamma=0.$ From (\ref{f}) $\gamma$ must be
zero, contradiction! This proves the corollary.
\end{proof}

The case that the image of the Gauss map is contained in a \emph{closed
}hemisphere of the sphere is more complicated. Recall that in the
$3-$dimensional Euclidean space, D. Hoffman, R. Osserman and R. Schoen proved
that if the image of a complete CMC surface is contained in a closed
hemisphere of $\mathbb{S}^{2}$ then the surface is a plane or a right circular
cylinder (Theorem 1 of \cite{HOS}). With similar suitable hypothesis on the
\emph{Grassmanian }image of the Gauss map of a complete surface in
$\mathbb{R}^{4}$, with non zero parallel mean curvature vector, they also
prove in \cite{HOS} that a surface is plane, a cylinder or a flat torus. A
similar result holds when the ambient space is $\mathbb{S}^{3}$ (Theorems 2
and 3 of \cite{HOS}).

We obtain here a theorem for a CMC compact hypersurface $M$ of $\mathbb{S}%
^{7}$ assuming that the image of the octonionic Gauss map of $M$ is contained
in a closed hemisphere of $\mathbb{S}^{6}$. We are not able to prove an
isometric rigidity theorem but only some topological rigidity (isometric
rigidity seems to be a difficult problem in higher dimensions. To our
knowledge no such a result was obtained for CMC hypersurfaces of
$\mathbb{R}^{n}$ or $\mathbb{S}^{n}$ for $n\geq5$).

Observe that in the results of \cite{HOS}, mentioned above, the Euler
characteristic of a CMC surface having the Gauss map contained in a closed
hemisphere is zero. We obtain the same conclusion on the Euler characteristic
of CMC hypersurfaces having the Gauss image contained in a closed hemisphere
of $\mathbb{S}^{6}$ (Corollary \ref{eu} below). In fact, we obtain some
stronger topological theorem related to the "size" of the image of the Gauss
map. To state precisely our results we need to introduce some facts.

Given a totally geodesic $5$-dimensional sphere $T$ of $\mathbb{S}^{6}$, let
$P_{T}$ be the hyperplane through the origin of $\mathbb{R}^{7}$ such that
$T=P_{T}\cap\mathbb{S}^{6}$. We shall say that $j$ totally geodesic
$5$-dimensional spheres $T_{1},\dots,T_{j}$ of $\mathbb{S}^{6}$, $1\leq
j\leq7$, are linearly independent, if the unit normal vectors orthogonal to
each the hyperplanes $P_{T_{1}},\dots,P_{T_{j}}$ are linearly independent in
$\mathbb{R}^{7}$. Note that $T_{1}$ divides $\mathbb{S}^{6}$ into two
connected components. The closure of these connected components are called
\emph{closed hemispheres} of $\mathbb{S}^{6};$ $T_{1}\cup T_{2}$ divides
$\mathbb{S}^{6}$ into $2^{2}$ connected components which closure are called
\emph{closed} \emph{quadrants}; $T_{1}\cup T_{2}\cup T_{3}$ into $2^{3}$
connected components determining \emph{closed} \emph{octants}. In general,
$T_{1}\cup\cdots\cup T_{j}$ divides $\mathbb{S}^{6}$ into $2^{j}$ connected
components which closures are called \emph{closed }$2^{j}$-\emph{orthants }of
$\mathbb{S}^{6}$ (see \cite{R}).

By a $k$-vector field $(X_{1},\dots,X_{k})$ on an $l$-dimensional manifold
$N$, $k\leq l$, we mean an ordered set of $k$ vector fields $X_{1},...,X_{k}$
of $N$ which are linearly independent at each point of $N.$ If a $k-$vector
field $(X_{1},...,X_{k})$ is defined for all but a finite number of points we
say that it is a $k-$vector field with finite singularities.

For reader's sake, we recall now the definition of the index of a $k$-vector
field $X=(X_{1},...,X_{k})$ with finite singularities, as done in \cite{Th}.
Let $p\in N$ be a singularity of $X$. Consider a simplicial triangulation of
$N$ such that any singularity of $X$ is contained in the interior of a
$l$-simplex and that $p$ belongs to the interior of a simplex $\sigma$. The
tangent bundle of $N$ restricted to $\sigma$ is isomorphic to the product
bundle $\sigma\times\mathbb{R}^{l}$, and we assume $N$ is an oriented manifold
and that this isomorphism preserves orientation. We may also assume that
$\sigma$ is a ball in $\mathbb{R}^{l}$ and that $(X_{1},\dots,X_{k})$ is a
$k$-vector field of $\mathbb{R}^{l}$. Recalling that the Stiefel manifold
$V_{l,k}$ of $\mathbb{R}^{l}$ is defined as the space of $k\times l$ matrices
with linearly independent rows, we have that $(X_{1},...,X_{k})(x)\in V_{l,k}$
for any $x\in\sigma$. Since $\partial\sigma$ is a topological $(l-1)$-sphere,
the homotopy class of the map $\partial\sigma\to V_{l,k}$, given by
$x\mapsto(X_{1},\dots,X_{k})(x)$ is an element of the $(l-1)$-homotopy group
$\pi_{l-1}\left(  V_{l,k}\right)  $ of $V_{l,k}$. This homotopy class is, by
definition, the \emph{index } $I_{(X_{1},\dots,X_{k})}(p)$ of $(X_{1}%
,\dots,X_{k})$ at $p$. We then define the index $I_{(X_{1},\dots,X_{k})}$ of
$(X_{1},\dots,X_{k})$ in $N$ by
\[
I_{(X_{1},\dots,X_{k})}=\sum_{\substack{p\in N \\(X_{1},\dots,X_{k}%
)(p)=0}}I_{(X_{1},\dots,X_{k})}(p).
\]

The following result is well known in Differential Topology (see \S 34.2 of
\cite{St}).

\begin{theorem}
\label{ind} Let $N$ be an $l$-dimensional manifold, $l \geq2$, and let $(
X_{1}, \dots, X_{k} )$ a $k$-vector field with finite singularities on $N$, $1
\leq k \leq l$. Then $I_{( X_{1}, \dots, X_{k} )} = 0$ if and only if there is
a $k$-field with no singularities on $N$ which coincides with $( X_{1}, \dots,
X_{k} )$ on the $(l - 2)$-skeleton of $N$.
\end{theorem}

We note that the existence of $k$-fields defined on the whole manifold or
$k$-fields with finite singularities and with index zero has strong influence
on the topology of the manifold (see \cite{Th}).

\begin{theorem}
\label{mao} Let $M$ be a compact and oriented immersed hypersurface of
constant mean curvature of $\mathbb{S}^{7}$ and let $\gamma: M \to
\mathbb{S}^{6}$ be the Gauss map of $M$. The following alternatives are equivalent:

\begin{itemize}
\item[(i)] The image $\gamma(M)$ is contained in a closed $2^{k}$-orthant of
$\mathbb{S}^{6}$, that is, the closure of a connected component of
$\mathbb{S}^{6}\setminus\left(  T_{1}\cup\cdots\cup T_{k}\right)  $, for some
linearly independent totally geodesic spheres $T_{1},\dots,T_{k}$ of
$\mathbb{S}^{6}$, $1\leq k\leq7.$

\item[(ii)]
\[
\gamma(M) \subset\bigcap_{j=1}^{k} T_{j}%
\]

\item[(iii)] Let $\{v_{1},\dots,v_{k}\}$ be a basis of
\[
\left(  \bigcap_{j=1}^{k}T_{j}\right)  ^{\perp}%
\]
and let $\mathcal{H}$ be the Lie algebra generated by the Killing fields
$V_{v_{1}},\dots,V_{v_{k}}$. Then $\mathcal{H}$ is a Lie subalgebra of the Lie
algebra of the isometry group of $M$.
\end{itemize}

Any of the above alternatives implies that $k\leq6$, that is, it does not
exist a compact connected oriented CMC hypersurface $M$ of $\mathbb{S}^{7}$
such that $\gamma(M)$ is contained in a closed $2^{7}$-orthant of
$\mathbb{S}^{6}$. Moreover, $\dim\operatorname*{Iso}(M)\geq k$ and the index
of any $k$-vector field with finite singularities on $M$ is zero.
\end{theorem}

\begin{proof}
We prove that (i) implies (ii). Since $\gamma(M)$ is contained in a closed
$2^{k}$ orthant of $\mathbb{S}^{6}$, there are $k$ linearly independents
vectors $v_{1},...,v_{k}$ such that $\left\langle \gamma,v_{i}\right\rangle
\geq0$, $1\leq i\leq k$. From (\ref{f}), with a fixed $i$, we obtain
\begin{equation}
\Delta\left(  \left\langle \gamma,v_{i}\right\rangle \right)  =\left\langle
\Delta\gamma,v_{i}\right\rangle =-\left(  6+\left\Vert A\right\Vert
^{2}\right)  \left\langle \gamma,v_{i}\right\rangle \leq0.\label{f2}%
\end{equation}
It follows that $\left\langle \gamma,v_{i}\right\rangle $ is superharmonic on
$M$. Since $M$ is compact, we have that $\left\langle \gamma,v_{i}%
\right\rangle $ is a constant. But then $\Delta\left\langle \gamma
,v_{i}\right\rangle =0$ and it follows from (\ref{f2}) that $\left\langle
\gamma,v_{i}\right\rangle =0$. Hence, given $x\in M$,
\[
0=\left\langle \gamma(x),v_{i}\right\rangle =\left\langle \Gamma_{x^{-1}%
}(\gamma(x)),\Gamma_{x^{-1}}(v_{i})\right\rangle =\left\langle \eta
(x),V_{v_{i}}(x)\right\rangle ,
\]
so that $V_{v_{i}}$ is a Hopf vector field of $\mathbb{S}^{7}$ tangent to
$M\ $, proving that $\gamma(M)\subset T_{i}$. Since this holds for all $1\leq
i\leq k$, it follows that (i) implies (ii).

\qquad As to the equivalence between (i) or (ii) and (iii), one only has to
note that if $V_{v_{1}},\dots,V_{v_{k}}$ are vector fields on $M$, then so is
the bracket of any two of them. Moreover, since the bracket of Killing fields
is another Killing field, the assertion related to $\mathcal{H}$ follows; in
particular $\dim\operatorname*{Iso}(M)\geq k$. Furthermore, since $V_{v_{1}%
},\dots,V_{v_{k}}$ are linearly independent, it follows from Theorem \ref{ind}
that the index of any $k$-vector field with finite singularities on $M$ is
zero. Finally, we have $k\leq6$, otherwise $\gamma(M)\subset\mathbb{S}%
^{6}\subset T_{1}\mathbb{S}^{7}\cong\mathbb{R}^{7}$, by (ii), would be
orthogonal to $7$ linearly independent vectors.
\end{proof}

Using the previous notation and definition of the index of a $k-$vector field
with finite singularities, note that when $k=1$, since $V_{l,1}$ has the
homotopy type of the sphere $\mathbb{S}^{l-1}$ and since $\pi_{l-1}%
(\mathbb{S}^{l-1})\cong\mathbb{Z}$, it follows that $I(X_{1})$ is an integer
which is well known to be the Euler characteristic of the manifold. Then, as a
consequence of Theorem \ref{mao}, we have:

\begin{corollary}
\label{eu}Let $M$ be a compact hypersurface of constant mean curvature of
$\mathbb{S}^{7}$. If the Gauss map $\gamma:M\rightarrow\mathbb{S}^{6}$ of $M$
is contained in a closed hemisphere of $\mathbb{S}^{6}$ then the Euler
characteristic of $M$ is zero.
\end{corollary}

\section{\label{cp33}The Gauss map of constant mean curvature hypersurfaces of
$\mathbb{CP}^{3}$}

\qquad An\emph{ }action of $\mathbb{S}^{1}=\left\{  e^{i\theta}\right\}
_{\theta\in\mathbb{R}}$ on $\mathbb{S}^{7}$ is a \emph{Hopf action} if the
associated vector field%
\[
X(x):=\left.  \frac{d}{d\theta}e^{i\theta}\left(  x\right)  \right\vert
_{\theta=0},\text{ }x\in\mathbb{S}^{7},
\]
is a Hopf vector field. The complex projective space $\mathbb{CP}^{3}$ with
the Fubini-Study metric is the quotient $\mathbb{S}^{7}/\mathbb{S}^{1}$ of
$\mathbb{S}^{7}$ under a Hopf action of $\mathbb{S}^{1},$ and with the metric
such that the projection $\pi:\mathbb{S}^{7}\rightarrow\mathbb{CP}^{3}$ is a
Riemannian submersion$.$

Recall that $\mathbb{S}^{1}=\left\{  e^{i\theta}\right\}  $ acts on
$\mathbb{S}^{7}$ by the octonionic multiplication $e^{i\theta}(x):=e^{i\theta
}\cdot x$ (it is indeed an action since $\cdot$ is associative in each
subspace of octonions generated by two elements (\cite{Ba})).

\begin{lemma}
The octonionic action of $\mathbb{S}^{1}$ on $\mathbb{S}^{7}$ is a Hopf action.
\end{lemma}

\begin{proof}
One may see that the vector field $W$ of $\mathbb{S}^{7}$ determined by the
octonionic action of $\mathbb{S}^{1}$ in $\mathbb{S}^{7}$ is
\[
W(p)=i\cdot p=\left(  i\right)  \left[  p\right]  ^{t}%
\]
where $\left(  i\right)  $ is the $8\times8$ and $\left[  p\right]  $ the
$1\times8$ matrix obtained writting $i$ and $p$ as linear combination of the
canonical basis of $\mathbb{R}^{8}$. We have that $\left(  i\right)  $ is a
Hopf vector field. Indeed,%

\[
\left(  i\right)  =AX_{0}A^{-1}%
\]
where $A$ is the orthogonal matrix%
\[
A=\left(
\begin{array}
[c]{cccc}%
B & 0 & 0 & 0\\
0 & B & 0 & 0\\
0 & 0 & B & 0\\
0 & 0 & 0 & B
\end{array}
\right)  ,\text{ }B=\left(
\begin{array}
[c]{cc}%
0 & 1\\
1 & 0
\end{array}
\right)
\]

\end{proof}

The \emph{Hopf symmetrization} of a vector field $X\in\Xi\left(
\mathbb{S}^{7}\right)  $ is a vector field $X^{h}\in\Xi\left(  \mathbb{S}%
^{7}\right)  $ defined by
\[
X^{h}(x)=\frac{1}{2\pi}\int_{0}^{2\pi}e^{-i\theta}\cdot X(e^{i\theta}\cdot
x)d\theta.
\]

We have

\begin{lemma}
\label{la}The Hopf symmetrization $X^{h}$ of a vector field $X$ of
$\mathbb{S}^{7}$ is invariant by the Hopf action, that is%
\[
X^{h}(e^{i\phi}\cdot x)=e^{i\phi}\cdot X^{h}(x)
\]
for all $x\in\mathbb{S}^{7}$ and for all $\phi.$
\end{lemma}

\begin{proof}
We have%
\begin{align*}
X^{h}(e^{i\phi}\cdot x)  &  =\frac{1}{2\pi}\int_{0}^{2\pi}e^{-i\theta}\cdot
X(e^{i(\theta+\phi)}\cdot x)d\theta\\
&  =\frac{e^{i\phi}}{2\pi}\cdot\int_{0}^{2\pi}e^{-iu}\cdot X(e^{i(u)}\cdot
x)du=e^{i\phi}\cdot X^{h}(x),\text{ }x\in\mathbb{S}^{7},\text{ }\phi
\in\mathbb{R}.
\end{align*}

\end{proof}

\begin{lemma}
\label{xhx}$X\in\Xi\left(  \mathbb{S}^{7}\right)  $ is Hopf invariant if and
only if $X^{h}=X$.
\end{lemma}

\begin{proof}
If $X^{h}=X$ then $X$ is Hopf invariant by Lemma \ref{la}. Conversely if
$X\left(  e^{i\phi}\cdot x\right)  =e^{i\phi}\cdot X(x)$ then
\[
X^{h}(x)=\frac{1}{2\pi}\int_{0}^{2\pi}e^{-i\theta}\cdot X(e^{i\theta}\cdot
x)d\theta=X(x),\text{ }x\in\mathbb{S}^{7}.
\]
\bigskip
\end{proof}

Note that a vector field $W$ of $\mathbb{S}^{7}$ which is invariant by the
Hopf action defines a vector field $Z$ in $\mathbb{CP}^{3}$. Indeed, if
$y=e^{i\theta}\cdot x,$ $x\in\mathbb{S}^{7},$ then
\[
d\pi_{y}\left(  W(y)\right)  =d\pi_{e^{i\theta}\cdot x}\left(  W(e^{i\theta
}\cdot x)\right)  =d\pi_{e^{i\theta}\cdot x}\left(  e^{i\theta}\cdot
W(x)\right)  =d\pi_{x}\left(  W(x)\right)  .
\]
We may then define%
\[
Z\left(  z\right)  =d\pi_{x}\left(  W(x)\right)  ,\text{ }z\in\mathbb{CP}%
^{3},
\]
where $x\in\pi^{-1}(z)$. Given $v\in T_{1}\mathbb{S}^{7}$ set
\[
W_{v}(x)=\frac{x\cdot v-i\cdot((i\cdot x)\cdot v)}{2},\text{ }x\in
\mathbb{S}^{7}.
\]
We claim that $W_{v}$ is a vector field of $\mathbb{S}^{7}.$ Indeed, given
$x\in\mathbb{S}^{7}$ the map $F_{x}:\mathbb{S}^{7}\rightarrow\mathbb{S}^{7}$
given by
\[
F_{x}(u)=-i\cdot\left(  \left(  i\cdot\left(  x\cdot u\right)  \right)  \cdot
u\right)  ,\text{ }u\in\mathbb{S}^{7},
\]
satisfies
\[
F_{x}(1)=x
\]
and
\[
d\left(  F_{x}\right)  _{1}(h)=x\cdot h-i\cdot((i\cdot x)\cdot h).
\]
Hence
\[
W_{v}(x)=d\left(  F_{x}\right)  _{1}\left(  \frac{v}{2}\right)  \in
T_{x}\mathbb{S}^{7}%
\]
for all $x\in\mathbb{S}^{7}.$ We also have $W_{v}(1)=v.$

\begin{lemma}
$W_{v}$ is a Hopf invariant vector field for any $v\in T_{1}\mathbb{S}^{7}.$
\end{lemma}

\begin{proof}
Let $v\in T_{1}\mathbb{S}^{7}.$ We prove that $W_{v}^{h}=W_{v}.$ The lemma
then follows from Lemma \ref{xhx}.
\begin{align*}
W_{v}^{h}(x)  &  =\frac{1}{2\pi}\int_{0}^{2\pi}e^{-i\theta}\cdot
W_{v}(e^{i\theta}\cdot x)d\theta\\
&  =\frac{1}{2\pi}\int_{0}^{2\pi}e^{-i\theta}\cdot\left(  \frac{\left(
e^{i\theta}\cdot x\right)  \cdot v-i\cdot((i\cdot\left(  e^{i\theta}\cdot
x\right)  )\cdot v)}{2}\right)  d\theta\\
&  =\frac{1}{4\pi}\int_{0}^{2\pi}e^{-i\theta}\cdot\left(  \left(  e^{i\theta
}\cdot x\right)  \cdot v\right)  d\theta+\frac{1}{4\pi}\int_{0}^{2\pi
}e^{-i(w)}\cdot\left(  \left(  e^{i(w)}\cdot x\right)  \cdot v\right)  dw\\
&  =\frac{1}{2\pi}\int_{0}^{2\pi}e^{-i\theta}\cdot\left(  \left(  e^{i\theta
}\cdot x\right)  \cdot v\right)  d\theta.
\end{align*}
Writing $e^{i\theta}=\cos\theta+i\sin\theta$ we have
\begin{align*}
W_{v}^{h}(x)  &  =\left(  x\cdot v\right)  \frac{1}{2\pi}\int_{0}^{2\pi}%
\cos^{2}\theta d\theta+\left(  \left(  i\cdot x\right)  \cdot v\right)
\frac{1}{2\pi}\int_{0}^{2\pi}\cos\theta\sin\theta d\theta\\
&  -\left(  i\cdot\left(  x\cdot v\right)  \right)  \frac{1}{2\pi}\int
_{0}^{2\pi}\sin\theta\cos\theta d\theta-i\cdot\left(  \left(  i\cdot x\right)
\cdot v\right)  \frac{1}{2\pi}\int_{0}^{2\pi}\sin^{2}\theta d\theta\\
&  =\frac{\left(  x\cdot v\right)  -i\cdot\left(  \left(  i\cdot x\right)
\cdot v\right)  }{2}=W_{v}(x)
\end{align*}
concluding with the proof of the lemma.
\end{proof}

Note that since the symmetrization is an average process it may happen to
$X^{h}$ be identically zero for a nonzero vector field $X$. But this is not
the case with $W_{v}$ if $v\neq0$ since $W_{v}^{h}(1)=v$. And indeed, defining
the following vector fields of $\mathbb{CP}^{3}$: $Z_{n}(z)=d\pi_{x}\left(
W_{e_{n+1}}\left(  x\right)  \right)  ,$ $1\leq n\leq6$, where $x\in\pi
^{-1}(z)$ and $\left\{  e_{n}\right\}  _{1\leq n\leq7}$ is the canonical
octonionic basis (see \cite{Ba}), we have the stronger fact:

\begin{proposition}
Let $\mathbb{CP}^{2}\subset\mathbb{CP}^{3}$ be\ the totally geodesic real
codimension $2$ complex projective space which is the cut locus of $\pi
(1)\in\mathbb{CP}^{3}.$ Then $Z_{1},~...,~Z_{6}\ $are Killing fields of
$\mathbb{CP}^{3}$ which are linearly independent at any point of
$\mathbb{CP}^{3}\backslash\mathbb{CP}^{2}.$
\end{proposition}

\begin{proof}
We first note $Z_{n}$ are Killing fields of $\mathbb{CP}^{3}$. Indeed, the
vector fields $W_{e_{n}}$ are Killing fields on $\mathbb{S}^{7}$ since they
are given as a sum of compositions of right and left octonionic translations.
Since $W_{e_{n}}$ commutes with the Hopf action it belongs to the Lie algebra
of the unitary group $U(4),$ which is the isometry group of $\mathbb{CP}^{3}$
and hence projects into a Killing field of $\mathbb{CP}^{3}$.

For proving that the $Z_{n}$ are linearly independent it is enough to show
that $W_{e_{n}}$ are orthogonal to the fibers of $\pi$ for $2\leq n\leq7$ and,
since then%
\[
\left(  \left\langle Z_{n-1}(\pi(x)),Z_{m-1}(\pi(x))\right\rangle _{2\leq
n,m\leq7}\right)  =\left(  \left\langle W_{e_{n}}(x),W_{e_{m}}(x)\right\rangle
_{2\leq n,m\leq7}\right)  ,
\]
that
\[
\det\left(  \left\langle W_{e_{n}}(x),W_{e_{m}}(x)\right\rangle _{2\leq
n,m\leq7}\right)  \neq0
\]
for all $x\in\mathbb{S}^{7}\backslash\mathbb{S}^{5}$ where $\mathbb{S}^{5}$ is
the totally geodesic codimension $2$ sphere
\[
\mathbb{S}^{5}=\left\{  \left(  0,0,a_{2},...,a_{8}\right)  \text{
%TCIMACRO{\TEXTsymbol{\vert} }%
%BeginExpansion
$\vert$
%EndExpansion
}a_{2}^{2}+...+a_{8}^{2}=1\right\}  .
\]

Setting%
\[
x=\sum_{n=0}^{7}a_{n}e_{n}%
\]
we have%
\begin{align*}
W_{e_{1}}(x)  &  =i\cdot x=X_{0}(x),\\
W_{e_{n}}(x)  &  =\frac{x\cdot e_{n}-i\cdot((i\cdot x)\cdot e_{n})}{2}%
\end{align*}

We may then see that%
\[
\left\langle X_{0}(x),W_{e_{n}}(x)\right\rangle =\left\langle W_{e_{1}%
}(x),W_{e_{n}}(x)\right\rangle =0,\text{ }2\leq n\leq7.
\]
Moreover, a calculation shows that
\begin{align*}
&  \det\left(  \,\left\langle W_{e_{n}}(x),W_{e_{m}}(x)\right\rangle _{2\leq
n,m\leq7}\right) \\
&  =\left(  a_{0}^{2}+a_{1}^{2}\right)  ^{4}\left(  a_{0}^{2}+a_{1}^{2}%
+a_{2}^{2}+a_{3}^{2}+a_{4}^{2}+a_{5}^{2}+a_{6}^{2}+a_{7}^{2}\right)  ^{2}\\
&  =\left(  a_{0}^{2}+a_{1}^{2}\right)  ^{4},
\end{align*}
concluding the proof of the proposition.
\end{proof}

Define a translation $\Gamma:T\left(  \mathbb{CP}^{3}\backslash\mathbb{CP}%
^{2}\right)  \rightarrow\mathbb{R}^{6}$ by setting%
\[
\Gamma_{x}\left(  v\right)  =\left(  \left\langle v,Z_{1}(x)\right\rangle
,...,\left\langle v,Z_{6}(x)\right\rangle _{6}\right)  ,\text{ }%
x\in\mathbb{CP}^{3}\backslash\mathbb{CP}^{2},\text{ }v\in T_{x}\mathbb{CP}%
^{3}.
\]
Any $v\in\mathbb{R}^{6}\backslash\{0\}$ determines a Killing field $Z_{v}$ on
$\mathbb{CP}^{3}\backslash\mathbb{CP}^{2},$ without singularities, by setting
$Z_{v}(x)=\Gamma_{x}^{-1}(v).$ We have%
\[
Z_{v}(x)=v_{1}Z_{1}(x)+...+v_{6}Z_{6}(x)
\]
if $v=\left(  v_{1},...,v_{6}\right)  .$ If $M$ is an orientable hypersurface
of $\mathbb{CP}^{3}\backslash\mathbb{CP}^{2}$ and $\eta$ a unit normal vector
field along $M$ the Gauss map $\gamma:M\rightarrow\mathbb{R}^{6}$ of $M$ is
defined by%
\[
\gamma(x)=\Gamma_{x}\left(  \eta(x)\right)  ,\text{ }x\in M.
\]
Using a similar proof to that of Theorem \ref{for}, we obtain:

\begin{theorem}
Let $M$ be an orientable hypersurface of $\mathbb{CP}^{3}\backslash
\mathbb{CP}^{2}$, $\eta$ a unit normal vector field along $M$ and
$\gamma:M\rightarrow\mathbb{R}^{6}$ the associated Gauss map. Then%
\[
\Delta\gamma(x)=5\Gamma_{x}\left(  \operatorname{grad}H(x)\right)  +\left(
8+\left\Vert A(x)\right\Vert ^{2}\right)  \gamma\left(  x\right)  ,\text{
}x\in M,
\]
where $H$ and $A$ are the mean curvature function and the second fundamental
form of $M.$ In particular, $M$ has constant mean curvature if and only if
$\gamma$ satisfies the vectorial PDE%
\[
\Delta\gamma=\left(  8+\left\Vert A(x)\right\Vert ^{2}\right)  \gamma.
\]

\end{theorem}

If one replaces a hemisphere by a half-space of $\mathbb{R}^{6},$ a quadrant
as one connected component of $\mathbb{R}^{6}\setminus\left(  P_{1}\cup
P_{2}\right)  $ where $P_{i}$ are linearly independent hyperplanes through the
origin of $\mathbb{R}^{6}$ (that is, the unit normal vector to $P_{i}$ are
linearly independent), an octant as a connected component of $\mathbb{R}%
^{6}\setminus\left(  P_{1}\cup P_{2}\cup P_{3}\right)  $ and so on, we have an
extension of Theorem \ref{mao} to $\mathbb{CP}^{3}$. Its statement is
completely similar to Theorem \ref{mao} and therefore omitted. We think it is
worthwhile however to state a corollary on the image of the Gauss map:

\begin{corollary}
Let $M$ be a compact oriented immersed hypersurface of $\mathbb{CP}^{3}$ and
let $\gamma:M\rightarrow\mathbb{R}^{6}$ be the Gauss map of $M$. Then
$\gamma(M)$ is not contained in a half space of $\mathbb{R}^{6}.$
\end{corollary}

\end{document}